\documentclass[12pt, letterpaper]{article}

\usepackage{color}
\usepackage{setspace}           
\usepackage{graphicx}           
\usepackage{amsmath, bbm, amsfonts, amsthm}            
\usepackage{marginnote}         
\usepackage{datetime}           
\usdate                         
\usepackage{enumitem}           
\usepackage{subfigure}          
\usepackage{rotating}           
\usepackage{hyperref}           
\usepackage{float}              
\usepackage[square,sort,comma,numbers]{natbib}

\usepackage{texintro}           
\usepackage{tocvsec2}           
\usepackage{bibentry}           

\usepackage[margin=1in]{geometry}

\makeatletter\let\chapter\@undefined\makeatother 

\newcommand{\E}{\mathbb{E}}
\newcommand{\cP}{\mathcal{P}}

\newcommand{\R}{\mathbb{R}}

\renewcommand{\P}{\mathbb{P}}
\newcommand{\e}{\epsilon}

\newcommand{\C}{\mathcal{C}}

\renewcommand{\H}{\mathbb{H}}

\renewcommand{\v}{\nu}

\renewcommand{\L}{\mathcal{L}}

\theoremstyle{plain}
\newtheorem{Theorem}{Theorem}[section]

\newtheorem{prop}[Theorem]{Proposition}

\theoremstyle{definition}
\newtheorem{Definition}[Theorem]{Definition}
\newtheorem{Example}[Theorem]{Example}

\theoremstyle{remark}
\newtheorem{Remark}[Theorem]{Remark}

\numberwithin{equation}{section}

\begin{document}

\nobibliography*                                
\setlist{noitemsep}                             


\title{Mean Field Game with Delay: a Toy Model}

\author{
Jean-Pierre Fouque
    \thanks{Department of Statistics \& Applied Probability, University of California, Santa Barbara, CA \ 93106-3110, e-mail: fouque@pstat.ucsb.edu. Work  supported by NSF grants DMS-1409434 and DMS-1814091.}
\and
Zhaoyu Zhang 
    \thanks{Department of Statistics \& Applied Probability, University of California, Santa Barbara, CA \ 93106-3110, e-mail: zhaoyu\_zhang@ucsb.edu}
}

\date{\today}


\maketitle

\begin{abstract}
We study a toy model of linear-quadratic mean field game with delay. We ``lift" the delayed dynamic into an infinite dimensional space, and recast the mean field game system which is made of a forward Kolmogorov equation and a backward Hamilton-Jacobi-Bellman equation. We identify the corresponding master equation. A solution to this master equation is computed, and 
we show that it provides an approximation to a Nash equilibrium of the finite player game.
\end{abstract}




\onehalfspacing



\section{Introduction}\label{sec: intro}
A linear quadratic stochastic game model of inter-bank borrowing and lending was proposed in \citep{Carmona_Fouque_Sun:2015}. In this model, each individual bank tries to minimize its costs by controlling its rate of borrowing or lending to a central bank with no obligation to pay back its loan. The finding is that, in  equilibrium, the central bank acts as a clearing house providing liquidity, and hence stability is enhanced. This model was extended in \citep{Carmona_Fouque_Mousavi_Sun:2018}, where a delay in the controls was introduced. The financial motivation is that banks are responsible for the past borrowing or lending, and need to make a repayment after a fixed time (the delay). In this model, the dynamics of the log-monetary reserves of the banks are described by stochastic delayed differential equations (SDDE). A closed-loop Nash equilibrium is identified by formulating the original SDDE in an infinite dimensional space formed by the state and the past of the control, and by solving the corresponding infinite dimensional Hamilton-Jacobi-Bellman (HJB) equation. For  general stochastic equations and control theory in infinite dimension,  we refer to \citep{Bensoussan_book:2007}, \citep{Fabbi_Gozzi_Swiech:2017}, and \citep{DaPrato_Zabczyk:2008}. 

In this paper, we study the  mean field game (MFG) corresponding to the model proposed in \citep{Carmona_Fouque_Mousavi_Sun:2018} as the number of banks goes to infinity. We identify the mean field game system, which is a system of coupled partial differential equations (PDEs). The forward Kolmogorov equation describes the dynamics of the joint law of current state and past control, and the backward HJB equation describes the evolution of the value function. Recently, J.-M. Lasry and P.-L. Lions introduced the concept of ``master equation" which contains all the information about the MFG. The  well-posedness of this master equation in presence of a common noise and convergence of the $N$-player system is analyzed in \citep{Cardali_Delarue_Lasry_Lions:2015} by a PDE approach. A probabilistic approach is proposed in 
\citep{Carmona_Delarue:2014} and \citep{Chassa_Crisan_Delarue:2015}. See also the two-volume book \citep{Carmona_Delarue_vol:2018} for a complete account of this approach. 

In this paper, the master equation for our delayed mean field game is derived, a solution is given explicitly, and we show that 
it is the limit of the closed-loop Nash equilibrium of 
the   $N$-player game system as $N\to\infty$.

The paper is organized as follows. In Section \ref{sec:gamedelay}, we briefly review the stochastic game model with delay presented in \citep{Carmona_Fouque_Mousavi_Sun:2018}. Then, in Section \ref{sec:MFG}, we construct the corresponding  mean field game system. In Section \ref{sec:mastereq}, we define derivatives with respect to probability measures in the space $\cP(\H)$ where $\H$ is the Hilbert space defined at the beginning of Section \ref{sec:construction}. In addition, we derive the master equation, and exhibit an explicit solution. Furthermore, in Section \ref{sec:convergence}, we show that this solution of the master equation is an approximation of order $1/N$ to the solution of the finite-player Nash system. Lastly, in Section \ref{sec:conclusion}, we compare the solution of the Nash system, the solution of the mean field game system, and the solution to the master equation.

\section{A differential game with delay}\label{sec:gamedelay}

\subsection{The model}

Let $\left(X_t^i, i =1, \cdots, N\right)$ represents the log-monetary reserves of the $N$ banks  at time t. At each time $t$, bank $i$ controls its rate of borrowing or lending $\alpha_t^i$, and it also needs to make a repayment after a fixed time $\tau$ such that $0 \leq \tau \leq T$, at a rate  denoted by $\alpha_{t-\tau}^i$. The dynamic of log-monetary reserves for each bank is given by 
\begin{equation}\label{org_system}
dX_t^i = (\alpha_t^i - \alpha_{t-\tau}^i) dt + \sigma dW_t^i,
\end{equation}
with deterministic initial conditions
\begin{equation}
X_0^i = \xi^i, \quad \mbox{ and } \alpha_s^i = \phi^i(s) \mbox{ for } s \in [-\tau, 0],
\end{equation}
where $W_t^i, \ i = 1, \dots, N$ are independent standard Brownian motions, and banks have the same volatility $\sigma > 0$.

Bank $i$ interacts with other banks by choosing its own strategy in order to minimize its cost functional $J^i(\alpha^i, \alpha^{-i})$, 
which involves the average of log-monetary reserves of all the other banks. The notation $\alpha^{-i}$ is a $(N-1)$ tuple of the $\alpha^j$ with $j \neq i$ and $j \in \{1, \cdots, N\}$, which
represents all other banks' control except bank $i$. The cost functional for bank $i \in \{1, \dots, N\}$ is given by:
\begin{equation}
J^i(\alpha^i, \alpha^{-i}) = \E \left[\int_0^T f_i(X_t, \alpha_t^i) dt + g_i(X_T)\right],
\end{equation}
where the running and terminal cost functions $f$ and $g$ are:
\begin{equation}\label{cost}
\begin{split}
& f_i(x, \alpha^i) = \frac{1}{2}(\alpha^i)^2 + \frac{\e}{2}(\bar{x} - x^i)^2, \mbox{ with } \bar{x} := \frac{1}{N} \sum_{k=1}^N x^k, \mbox{ and } \e > 0,\\
& g_i(x) = \frac{c}{2}(\bar{x} - x^i)^2, \ c \geq 0.
\end{split}
\end{equation}

\subsection{Construction of a Nash equilibrium}\label{sec:construction}
In order to apply the dynamic programming principle to identify a closed-loop Nash equilibrium, we have to enlarge the state space by including the path of past controls, which lie in $\H:= L^2([-\tau, 0]; \R)$, the Hilbert space of square integrable real functions defined on $[-\tau,0]$, and write an infinite dimensional representation for our system. This evolution equation approach was initiated in \citep{Vinter_Kwong:1981} under a deterministic control setting, and later was generalized in \citep{Gozzi_Marinelli:2006} to a stochastic control problem. 

Given $z \in \R \times \H$, $z_0 \in \R$, and $z_1 \in \H$ will denote the two components of the product space $\R \times \H$. The inner product on $\R \times \H$ will be denoted by $\langle \cdot, \cdot \rangle$, and it is defined by
\begin{equation}
\langle z, \tilde{z} \rangle = z_0 \tilde{z}_0 + \int_{-\tau}^0 z_1(s) \tilde{z}_1(s) ds.
\end{equation}
Therefore, the new state is denoted by $Z^i_t = (Z^i_{0,t}, Z^i_{1,t}(s)), \ s \in [-\tau, 0]$, which corresponds to $(X_t^i, \alpha^i_{t-\tau-s})$ in the notation of the original system \eqref{org_system}. 

Bank $i$ tries to minimize its cost functional $J^i(\alpha^i, \alpha^{-i})$ defined by
\begin{equation}
J^i(t,z,\alpha^i, \alpha^{-i}) = \E \left[ \int_t^T f_i(Z_{0,s}, \alpha_s^i) ds + g_i(Z_{0,T}) | Z_t = z \right].
\end{equation}
After all other players $j \neq i$ have chosen their optimal strategies which minimize their cost functionals, player $i$'s value function $V^i(t,z)$ is defined by
\begin{equation*}
V^i(t, z) = \inf_{\alpha^i} J^i(t,z, \alpha^i, \alpha^{-i}).
\end{equation*}
By dynamic programming principle, the value function $V^i(t,z)$ must satisfy the following infinite dimensional HJB equation (see \citep{Fabbi_Gozzi_Swiech:2017} Chapter 2 for details):
\begin{multline}\label{hjb}
\partial_t V^i(t,z) + \frac{1}{2}Tr(G^* G \partial_{zz} V^i(t,z)) + \sum_{k=1}^N \langle Az^k, \partial_{z^k} V^i(t,z) \rangle \\ 
+ \inf_{\alpha^i} \left[\sum_{k=1}^N \langle B\alpha^k, \partial_{z^k} V^i(t,z) \rangle + f_i(z_0, \alpha^i) \right] = 0,
\end{multline}
with terminal condition $V^i(T,z) = \frac{c}{2}(\bar{z}_0 - z_0^i)^2$, where the operator $A: D(A) \subset \R \times \H \to \R \times \H$ is defined as 
\begin{equation*}
A: (z_0, z_1(s)) \to \left(z_1(0), -\frac{d z_1(s)}{ds} \right) \ a.e., \ s \in [-\tau, 0],
\end{equation*}
and its domain is 
$
D(A) = \{(z_0, z_1 (\cdot)) \in \R \times \H: z_1(\cdot) \in W^{1,2} ([-\tau, 0]; \R), \ z_1(-\tau) = 0\}
$.

The adjoint of $A$ is $A^*: D(A^*) \subset \R \times \H \to \R \times \H$ and is defined by
\begin{equation*}
A^*: (z_0, z_1(s)) \to \left(0, \frac{dz_1(s)}{ds} \right) \ a.e., \   s \in [-\tau, 0],
\end{equation*}
with domain
$
D(A^*) = \{(z_0, z_1(\cdot) ) \in \R \times \H: z_1(\cdot) \in W^{1,2}([-\tau, 0]; \R), \ z_0 = z_1(0)\}
$.

The operator $B: \R \to \R \times \H$ is defined by
\begin{equation*}
B: u \to (u, -\delta_{-\tau} (s) u), \ s \in [-\tau, 0],
\end{equation*}
where $\delta_{-\tau}(\cdot)$ is the Dirac measure at $-\tau$.

The adjoint of $B$ is $B^*: \R \times \H \to \R$ given by 
$$B^*: (z_0, z_1 (s)) \to z_0 - z_1 (-\tau).$$
The operator $G: \R^N \to \R^N \times \H^N$ is defined by
\begin{equation*}
 G: z_0 \to (\sigma z_0, 0).
\end{equation*} 

The infinite dimensional representation of the original system \eqref{org_system} is given by
\begin{equation}
\begin{aligned}
& dZ_t^i = (AZ_t^i + B \alpha_t^i) dt + G dW_t, \ 0 \leq t \leq T,\\
& Z_0^i = (\xi^i, \phi^i(s)) \in \H.\\
\end{aligned}
\end{equation}

By minimizing the Hamiltonian in \eqref{hjb}, the infimum can be computed, so that the optimal control is attained at
\begin{equation}
\hat{\alpha}^i = -\langle B, \partial_{z^i} V^i \rangle = - \left(\partial_{z_0^i} V^i - [\partial_{z_1^i} V^i](-\tau) \right).
\end{equation}
Assuming that each player follows its own optimal strategy $(\hat{\alpha}^i)_{1 \leq i \leq N}$, which forms a Nash equilibrium, the corresponding value function follows the HJB equation
\begin{eqnarray}
&\hskip -6cm \partial_t V^i + \frac{1}{2}Tr(G^* G \partial_{zz} V^i) + \sum_{k=1}^N \langle A z^k, \partial_{z^k} V^i \rangle \nonumber\\
&\hskip 2cm - \sum_{k \neq i} \left( B^* \partial_{z^k} V^i \right) \cdot \left(  B^*\partial_{z^k} V^k\right)
 - \frac{1}{2} (B^* \partial_{z^i} V^i)^2 + \frac{\e}{2} (\bar{z}_0 - z_0^i)^2 = 0.
\end{eqnarray}
After applying the definitions of the operators $A, B$ and $Q$, the HJB equation for player $i$ becomes:
\begin{multline}\label{hjbnplayer}
\partial_t V^i + \sum_{k=1}^N \frac{1}{2} \sigma^2 \partial_{z_0^k z_0^k} V^i + \sum_{k=1}^N \int_{-\tau}^0 z_{1}^k  \frac{d}{ds} (\partial_{z_{1}^k} V^i) ds \\
-\sum_{k\neq i}^N \left(\partial_{z_0^k} V^k - [\partial_{z_{1}^k} V^k] (-\tau) \right) \left(\partial_{z_0^k} V^i - [\partial_{z_{1}^k} V^i] (-\tau) \right)\\
- \frac{1}{2} \left( \partial_{z_0^i} V^i - [\partial_{z_{1}^i} V^i] (-\tau)\right)^2 + \frac{\e}{2}(\bar{z}_0 - z_0^i)^2 = 0.
\end{multline}

As shown in \citep{Carmona_Fouque_Mousavi_Sun:2018}, a solution of the system \eqref{hjbnplayer} can be found in the form
\begin{multline}\label{vi}
V^i(t,z) = E_0(t) (\bar{z}_0 - z_0^i)^2 - 2(\bar{z}_0 - z_0^i) \int_{-\tau}^0 E_1(t, -\tau - s) (\bar{z}_{1} - z_{1}^i) ds\\
+ \int_{-\tau}^0 \int_{-\tau}^0 E_2(t, -\tau - s, -\tau - r) (\bar{z}_{1} - z_{1}^i)(\bar{z}_{1} - z_{1}^i) ds dr + E_3(t), 
\end{multline}
for some deterministic functions $E_0(t)$, $E_1(t,s)$, $E_2(t,s,r)$, and $E_3(t)$ satisfying the following PDEs
\begin{equation}
\begin{split}\label{epde1}
& \frac{d E_0(t)}{dt} + 2 \left(\frac{1}{N^2}- 1\right) (E_0(t) + E_1(t, 0))^2 + \frac{\e}{2}= 0,\\
& \frac{\partial E_1(t, s)}{\partial t} - \frac{\partial E_1(t, s)}{\partial s} +2 \left(\frac{1}{N^2} - 1 \right) (E_0(t) + E_1(t,0)) (E_1(t, s) + E_2(t, s, 0))  = 0,\\
& \frac{\partial E_2(t, s, r)}{\partial t} - \frac{\partial E_2(t, s, r)}{s} - \frac{\partial E_2(t, s, r)}{r} \\
 & \hskip 3cm + 2 \left(\frac{1}{N^2}-1\right) (E_1(t, s) + E_2(t, s, 0))(E_1(t,r) + E_2(t,r,0)) = 0,\\
& \frac{dE_3(t)}{dt} + (1-\frac{1}{N})\sigma^2  E_0(t)  = 0,
\end{split}
\end{equation}
with boundary conditions: $\forall t \in [0,T]$ and $\forall s,r \in [-\tau, 0]$,
\begin{equation} 
\begin{split}\label{boundary}
& E_0(T) = \frac{c}{2}, \quad E_1(T,s) = 0, \quad E_2(T,s, r) = 0, \quad E_2(t,s,r) = E_2(t,r,s),\\
& E_1(t,-\tau) = -E_0(t), \quad E_2(t,s,-\tau) = -E_1(t,s), \quad E_3(T) = 0.\\
\end{split}
\end{equation}
This set of PDEs \eqref{epde1} with boundary conditions \eqref{boundary} admits a unique solution as shown in \citep{Vinter_Kwong:1981},
and the optimal strategies take the integral form
\begin{multline}
\hat{\alpha}_t^i = 2\left(1-\frac{1}{N} \right) \Bigg[\left( E_1(t,0) + E_0(t)\right) (\bar{z}_0 - z_0^i) \\
 - \int_{-\tau}^0 (E_2(t, -\tau-s, 0) + E_1(t, -\tau-s)) (\bar{z}_{1} - z_1^i) ds\Bigg].
\end{multline}

\section{The mean field game system}\label{sec:MFG}

The mean field game theory describes the structure of a game with infinite many indistinguishable players. All players are rational, i.e., each player tries to minimize their cost against the mass of other players. This assumption implies that the running cost and terminal cost in \eqref{cost} only depend on i-th player's state $z_0^i$ and the empirical distribution of  $(z_0^j)_{j \neq i}$. Denoting this empirical distribution by
$$\mu_0^i = \frac{1}{N-1} \sum_{j \neq i} \delta_{z_0^j},$$ 
 these costs, as in \eqref{cost}, can be re-written as
\begin{equation}
\begin{split}
& f_i(z_0, \alpha^i) 
= \frac{1}{2} (\alpha^i)^2 + \frac{\e}{2}(\bar{z}_0 - z_0^i)^2\\
& \quad \quad \quad 
= \frac{1}{2}(\alpha^i)^2 + \frac{\e}{2} \left( 1-\frac{1}{N}\right)^2 \left( \int_\R y_0 d \mu_0^i(y_0) - z_0^i \right)^2 := f(z_0^i, \mu_0^i, \alpha^i) ,\\
& g_i(z_0) = \frac{c}{2} \left( 1-\frac{1}{N}\right)^2 \left( \int_\R y_0 d \mu_0^i (y_0)- z_0^i \right)^2 := g(z_0^i, \mu_0^i) .
\end{split}
\end{equation}

As the number $N$ of players  goes to $\infty$, the joint empirical distribution of the states and past controls $Z^j_t=(Z^j_{0,t},Z^j_{1,t})$
\[
\nu^i_t:= \frac{1}{N-1} \sum_{j \neq i} \delta_{(Z^j_{0,t},Z^j_{1,t})},
\]
with marginals
\[
\mu_{0,t}^i = \frac{1}{N-1} \sum_{j \neq i} \delta_{Z_{0,t}^j},\, \mu_{1,t}^i = \frac{1}{N-1} \sum_{j \neq i} \delta_{Z_{1,t}^j},
\]
converges to a deterministic limit denoted by $\nu(t)$  (with marginals denoted by $\mu_0(t)$ and $\mu_1(t)$). Here, we assume that, at time $0$,  $\nu^i_0$ satisfies the LLN (for instance with i.i.d. $Z_0^j$), and that the propagation of chaos  property holds. A full justification of this property would involve generalizing the result in Section 2.1 of \citep{Carmona_Delarue:2014} to an infinite dimensional setting in order to take into account the past of the controls. This is highly technical but intuitively sound. A complete proof is beyond the scope of this paper.

In the limit, a single representative player tries to minimize his cost functional, and, dropping the index $i$,  his value function is defined as
\begin{equation}
V(t,z) = \inf_{(\alpha_s)_{t \leq s \leq T}} \E\left[\int_t^T f(s, Z_{0,s}, \mu_{0}(s), \alpha_s)ds + g(Z_{0,T}, \mu_{0}(T)) | Z_t = z\right],
\end{equation}
subject to
\begin{equation}\label{sde}
dZ_t = (AZ_t + B \alpha_t)dt + G dW_t.
\end{equation}

The HJB equation for the value function $V(t, z)$ reads
\begin{multline}\label{infHJB}
\partial_t V + \frac{1}{2}Tr (G^* G \partial_{zz} V)+ \langle AZ, \partial_z V \rangle \\
+ \inf_{\alpha} \left\{ \langle B \alpha, \partial_z V \rangle + \frac{1}{2} \alpha^2 + \frac{\e}{2}\left(\int_{\R} y_0 d\mu_{0}(y_0) - z_0 \right)^2 \right\} = 0,
\end{multline}
with terminal condition 
$
V(T, z) = \frac{c}{2}(\int_{\R} y_0 d \mu_{0}(y_0) - z_0)^2.
$
Then, we minimize in $\alpha$ to get 
\begin{equation}\label{alphahat}
\hat{\alpha}_t = -\left(\partial_{z_0} V - [\partial_{z_1} V](-\tau) \right).
\end{equation} 
After plugging it into \eqref{infHJB}, our backward HJB equation reads:
\begin{equation}\label{HJB}
\begin{split}
& \partial_t V +  \frac{1}{2} \sigma^2 \partial_{z_0 z_0} V + \int_{-\tau}^0 z_{1}  \frac{d}{ds} (\partial_{z_{1}} V) ds 
-\frac{1}{2} \left( \partial_{z_0} V - [\partial_{z_{1}} V] (-\tau)\right)^2 \\
& \hskip 8cm
+ \frac{\e}{2}\left(\int_{\R} y_0 d \mu_{0}(y_0) - z_0 \right)^2 = 0,\\
& V(T, z) = \frac{c}{2}\left(\int_{\R} y_0 d \mu_{0}(y_0) - z_{0} \right)^2.
\end{split}
\end{equation}

Next, since we ``lift" the original non-Markovian optimization problem into a infinite dimensional Markovian control problem, we are able to characterize the corresponding generator for \eqref{sde}, which is denoted by $\L_t$, 
\begin{equation}
\L_t \varphi(z) = \langle (AZ + B \hat{\alpha}_t), \partial_z \varphi \rangle + \frac{1}{2} Tr (G^*G \partial_{zz} \varphi),
\end{equation}
where $\varphi$ is a smooth function and the time dependency comes from  $\hat{\alpha}_t$ given by \eqref{alphahat}.
The derivation of  the adjoint $\L_t^*$ of $\L_t$ is given in  Appendix \ref{appendix}. Consequently, the forward Kolmogorov equation for the distribution $\nu(t)$ reads
\begin{equation}\label{fp}
\begin{split}
& \partial_t \nu = \int_{-\tau}^0 \partial_{z_1} \left(\frac{d}{ds} z_1 \nu \right) ds -\int_{-\tau}^0 \partial_{z_1} (z_1 \nu)(\delta_0(s) - \delta_{-\tau}(s) ) ds 
+  \partial_{z_0} \{(\partial_{z_0} V - [\partial_{z_1}V](-\tau)) \nu\} \\
& \quad \quad - \int_{-\tau}^0 \partial_{z_1} \{(\partial_{z_0} V - [\partial_{z_1} V] (-\tau) ) \nu \} \delta_{-\tau}(s) ds + \frac{1}{2} \sigma^2 \partial_{z_0 z_0} \nu, \\
& \nu(0) = \P(\xi, \phi(s)_{s \in [-\tau, 0]}).
\end{split}
\end{equation}

Combining \eqref{HJB} with \eqref{fp}, we obtain the mean field game system. To solve this, We make the following ansatz for the value function
\begin{multline}\label{ansatzmfg}
V(t,z) = E_0(t) (m_0 - z_0)^2 - 2(m_0- z_0) \int_{-\tau}^0 E_1(t, -\tau - s)  (m_1 - z_{1}) ds\\
+ \int_{-\tau}^0 \int_{-\tau}^0 E_2(t, -\tau - s, -\tau - r) (m_1 - z_{1})(m_1- z_{1}) ds dr + E_3(t).
\end{multline}	
where we denote the mean of state $m_0 := \int_{\R} z_0 d\mu_0(z_0)$, and the mean of past control  $m_1 := \int_{\H} z_1 d \mu_1(z_1)$. Plugging \eqref{ansatzmfg} into \eqref{fp}, multiplying both sides of \eqref{fp} by $z_0$, and integrating over $\R \times \H$,  we have
\begin{equation}
\begin{aligned}
& \int_{\R \times \H} z_0 \partial_t \nu dz = \int_{\R \times \H} z_0 \int_{-\tau}^0 \partial_{z_1} \left( \frac{d}{ds} z_1 \nu \right) ds  dz
- \int_{\R \times \H} z_0 \int_{-\tau}^0 \partial_{z_1} (z_1 \nu) (\delta_0 (s) - \delta_{-\tau} (s)) ds  dz \\
& \quad +  \int_{\R \times \H} z_0  \partial_{z_0} \{(\partial_{z_0} V - [\partial_{z_1}V](-\tau)) \nu\} dz
- \int_{\R \times \H} z_0 \int_{-\tau}^0 \partial_{z_1} \{(\partial_{z_0} V - [\partial_{z_1} V] (-\tau) ) \nu \} \delta_{-\tau}(s) ds dz \\
& \quad + \int_{\R \times \H} z_0 \frac{1}{2} \sigma^2 \partial_{z_0 z_0} \nu dz.\\
\end{aligned}
\end{equation}
After integration by parts, we obtain
\begin{equation}
\partial_{t} m_{0} 
= \int_{\R \times \H} \left\{ \partial_{z_0} V - [\partial_{z_1}V](-\tau) \right\} \nu dz = 0,
\end{equation}
as can be seen directly using \eqref{ansatzmfg}. 

Similarly, plugging \eqref{ansatzmfg} to \eqref{fp}, multiplying both sides of \eqref{fp} by $z_1$, and integrating over $\R \times \H$,  we get
\begin{equation}
\begin{aligned}
& \int_{\R \times \H} z_1 \partial_t \nu dz = \int_{\R \times \H} z_1 \int_{-\tau}^0 \partial_{z_1} \left( \frac{d}{ds} z_1 \nu \right) ds  dz
- \int_{\R \times \H} z_1 \int_{-\tau}^0 \partial_{z_1} (z_1 \nu) (\delta_0 (s) - \delta_{-\tau} (s)) ds  dz \\
& \quad +  \int_{\R \times \H} z_1  \partial_{z_0} \{(\partial_{z_0} V - [\partial_{z_1}V](-\tau)) \nu\} dz
- \int_{\R \times \H} z_1 \int_{-\tau}^0 \partial_{z_1} \{(\partial_{z_0} V - [\partial_{z_1} V] (-\tau) ) \nu \} \delta_{-\tau}(s) ds dz \\
& \quad + \int_{\R \times \H} z_1 \frac{1}{2} \sigma \partial_{z_0 z_0}^2 \nu dz.\\
\end{aligned}
\end{equation}
By integration by parts, we deduce
\begin{equation}
\begin{aligned}
& \partial_{t} m_{1} = -\int_{\R \times \H} \int_{-\tau}^0 \frac{d}{ds} z_1 \nu ds dz + \int_{\R \times \H} \int_{-\tau}^0 z_1 \nu (\delta_0(s) - \delta_{-\tau}(s)) ds dz \\
 & \quad \quad \quad \quad \quad +  \int_{\R \times \H}  \left\{ \partial_{z_0} V - [\partial_{z_1} V] (-\tau) \right\} \nu  dz \\
 & \quad \quad \quad =  0.
\end{aligned}
\end{equation}

Now we are ready to verify the ansatz \eqref{ansatzmfg}. We first compute the derivative of the ansatz,
\begin{equation}
\begin{split}
& \partial_t V = \frac{d E_0(t)}{dt} (m_0 - z_0)^2 - 2(m_0 - z_0) \int_{-\tau}^0 \frac{\partial E_1 (t, -\tau -s )}{\partial t}  (m_{1} - z_{1}) ds \\
& \quad \quad + \int_{-\tau}^0 \int_{-\tau}^0 \frac{\partial E_2 (t, -\tau-s, -\tau - r)}{\partial t} (m_{1} - z_{1}) (m_{1} - z_{1}) dsdr + \frac{d E_3 (t)}{dt}, \\
& \partial_{z_0} V = -2 E_0(t) (m_0 - z_0) + 2 \int_{-\tau}^0 E_1(t, -\tau-s) (m_{1} - z_{1}) ds, \\
& \partial_{z_1} V = 2E_1(t, -\tau- s)(m_0 - z_0) - 2 \int_{-\tau}^0 E_2(t, -\tau-s,-\tau-r) (m_{1} - z_{1}) dr, \\
& \partial_{z_0 z_0} V = 2 E_0(t). 
\end{split}
\end{equation}
Then, we plug the ansatz \eqref{ansatzmfg} into \eqref{hjb}, and by collecting $(m_0 - z_0)^2$ terms, $(m_0 - z_0)(m_1 - z_1)$ terms, $(m_{1} - z_{1})^2$ terms, and constant terms, we obtain the following system of PDEs:
\begin{equation}\label{epde2}
\begin{aligned}
& \frac{d E_0(t)}{dt} -2 (E_0(t) + E_1(t, 0))^2 + \frac{\e}{2}= 0,\\
& \frac{\partial E_1(t, s)}{\partial t} - \frac{\partial E_1(t, s)}{\partial s} -2 (E_0(t) + E_1(t,0)) (E_1(t, s) + E_2(t, s, 0))  = 0, \\
& \frac{\partial E_2(t,s,r)}{\partial t} - \frac{\partial E_2(t, s, r)}{s} - \frac{\partial E_2(t,s, r)}{r}\\ 
 & \hskip 3cm -2 (E_1(t, s) + E_2(t, s, 0))(E_1(t,r) + E_2(t,r,0)) = 0, \\
& \frac{dE_3(t)}{dt} + 	\sigma^2  E_0(t)  = 0,
\end{aligned}
\end{equation}
with boundary conditions
\begin{equation}\label{boundary2}
\begin{split}
& E_0(T) = \frac{c}{2}, \quad E_1(T,s) = 0, \quad E_2(T,s, r) = 0, \quad E_2(t,s,r) = E_2(t,r,s),\\
& E_1(t,-\tau) = -E_0(t), \quad E_2(t,s,-\tau) = -E_1(t,s), \quad E_3(T) = 0.
\end{split}
\end{equation}
As for (\ref{epde1}--\ref{boundary}), the system (\ref{epde2}--\ref{boundary2})
admits a unique solution.

\section{The master equation}\label{sec:mastereq}

\subsection{Derivatives}
The master equation for this delayed game lies in an infinite dimensional space, and it requires a notion of derivatives in the space of measures in $\cP(\H)$.  

The set $\cP(\H)$ of probability measure on $\H$ is endowed with Monge-Kantorovich distance
\begin{equation}
{\textbf d}_{MK}(\mu_1, \mu_1') = \sup \left\{ \left \| \int_{\H} f(z) d(\mu_1 - \mu_1') (z) \right\|_{\H}: f \in Lip_1(\H)\right\},
\end{equation}
where $Lip(\H)$ is the collection of real-valued Lipschitz functions on $\H$ with Lipschitz constant 1.

\begin{Definition}
We say that $F: \cP(\H) \to \H$ is $\C^1$ if there exists an operator $\frac{\delta F}{\delta \v}: \cP (\H) \times \H \to \H$ such that for any $\mu_1$ and $\mu_1' \in \cP(\H)$
\begin{equation}
\lim_{\e \to 0^+} \frac{F(\mu_1 + \e (\mu_1' - \mu_1)) - F(\mu_1)}{\e} = \int_{\H} \frac{\delta F}{\delta \mu_1}(\mu_1, y_1) d (\mu_1' - \mu_1) (y_1).
\end{equation}

\end{Definition}

\begin{Definition}
If $\frac{\delta F}{\delta \mu_1} (\mu_1, y_1)$ is of class $\C^1$ with respect to $y_1$, the marginal derivative $D_{\mu_1} F: \cP(\H) \times \H \to \H$ is defined in the sense of Fr\'echet derivative:
\begin{equation}
D_{\mu_1} F(\mu_1, y_1) := D_{y_1}\frac{\delta F}{\delta \mu_1} (\mu_1, y_1).
\end{equation}
\end{Definition}

\begin{Remark}
Usually we will encounter a map $U:\cP(\H) \to \R$. In this case, $U$ can be expressed in a form of composition $\tilde{U}\circ F$, where $\tilde{U}: \H \to \R$, and $F: \cP(\H) \to \H$, i.e., $U = (\tilde{U} \circ F)(\mu_1)$. 

If $\frac{\delta F}{\delta \mu_1}$ is $\C^1$ with respect to $y_1$, and $\tilde{U}$ is Fr\'echet differentiable, then $\frac{\delta U}{\delta \mu_1}: \cP(\H) \times \H \to \H$, and $D_{\mu_1} U: \cP(\H) \times \H \to \H$ are defined by
\begin{equation}
\frac{\delta U}{\delta \mu_1}(\mu_1, y_1):= (D_F \tilde{U}) \left( \frac{\delta F}{\delta \mu_1} \right), \mbox{ and } D_{\mu_1} U(\mu_1, y_1) := \left( D_F \tilde{U} \right) \left( D_{\mu_1} F \right).
\end{equation}

\end{Remark}

\begin{Example}
Suppose $U(\mu_1) = \int_{-\tau}^0 \int_{\H} g(x_1(s)) d \mu_1 (x_1) ds$, where $g: \H \to \H$ is Fr\'echet differentiable. Then $U(\mu_1)$ can be written as $\tilde{U}[F(\mu_1)](s)$, where $\tilde{U}[F] = \int_{-\tau}^0 F(s) ds$, and $F(\mu_1) = \int_{\H} g(x_1(s)) d \mu_1 (x_1)$.
Then 
\begin{equation*}
F(\mu_1 + \e(\mu_1' - \mu_1)) = \int_{\H} g(x_1(s)) d (\mu_1 + \e (\mu_1' - \mu_1)).
\end{equation*}
So \begin{equation*}
 \frac{F(\mu_1 + \e (\mu_1' - \mu_1)) - F(\mu_1)}{\e} = \int_{\H} g(x_1(s)) d (\mu_1' - \mu_1).
 \end{equation*}
Then 
\begin{equation*}
\frac{\delta F}{\delta \mu_1}(\mu_1, y_1) = g(y_1), \quad \mbox{ and }
D_{\mu_1}F(\mu_1, y_1) = D_{y_1}g(y_1).
\end{equation*}
Since $D_F\tilde{U}[F] = 1$, we have 
\begin{equation*}
\frac{\delta U}{\delta \mu_1}(\mu_1, y_1) = g(y_1) \mbox{ and }D_{\mu_1} U(\mu_1, y_1) = D_{y_1}g(y_1).
\end{equation*}

\end{Example}

\subsection{The master equation}

\begin{Theorem}
For any $(t_0, \nu_0) \in [0,T] \times \cP(\R \times \H)$, we define 
\begin{equation}
U(t_0, \cdot, \nu_0) := V(t_0, \cdot),
\end{equation}
where $(V, \nu)$ is a classical solution to the system of forward-backward equations \eqref{HJB} and \eqref{fp}, with initial condition $\nu(t_0) = \nu_0$, and terminal condition $V(T,z) = \frac{c}{2} (\int_\R y_0 d\mu_{0}(y_0) - z_{0})^2$, respectively. Then $U $ must satisfy the following master equation 
\begin{multline}\label{master}
\partial_t U (t, z_0, z_1, \nu)+ \frac{1}{2} \sigma^2 \partial_{z_0 z_0} U (t, z_0, z_1, \nu) + \frac{1}{2} \sigma^2 \int_\R \partial_{y_0}D_{\mu_0} U(t, z_0, z_1, \nu, y_0) d\mu_0 (y_0)\\ 
+ \int_{-\tau}^0 z_{1} \frac{d}{ds} \partial_{z_1} U (t, z_0, z_1, \nu) ds
+ \int_{-\tau}^0 \int_\H y_1 \frac{d}{ds} \left[ D_{\mu_1} U (t, z_0, z_1, \nu, y_1)\right](s) d \mu_1 (y_1) ds \\
- \int_{\R \times \H} \left(\partial_{y_0} U(t, y_0, y_1, \nu) - [\partial_{y_1} U (t, y_0, y_1, \nu)] (-\tau)\right) D_{\mu_0} U(t, z_0, z_1, \nu, y_0) d\nu(y)\\
 + \int_{\R \times \H} \left(\partial_{y_0} U (t, y_0, y_1, \nu) -  [\partial_{y_1} U (t, y_0, y_1, \nu)] (-\tau)\right) [D_{\mu_{1}} U (t, z_0, z_1, \nu, y_1)](-\tau) d \nu(y)\\
 - \frac{1}{2}(\partial_{z_0} U(t, z_0, z_1, \nu) - [\partial_{z_{1}} U (t, z_0, z_1, \nu)](-\tau))^2 
 + \frac{\e}{2} \left(\int_\R y_0 d\mu_0(y_0) - z_0 \right)^2 = 0,
\end{multline}
where $\mu_0 $ and $\mu_1$ are the marginal law for $Z_0$ and $Z_1$ respectively.

\end{Theorem}

\begin{proof}
%

For any $h \in [0, T-t_0]$, $V(t_0 + h, \cdot) = U(t_0 + h, \cdot, \nu(t_0 + h))$. Then
\begin{equation}\label{a}
\begin{aligned}
& \partial_t V(t_0, z) \\
= & \partial_t U(t_0, z, \nu_0) + \int_{\R \times \H} \frac{\delta U}{\delta \nu} (t_0, z, \nu, y) \partial_t \nu(t_0, y) dy\\
= & \partial_t U(t_0, z, \nu_0) + \int_{\R \times \H} \frac{\delta U}{\delta \nu} (t_0, z, \nu, y) 
\left(\int_{-\tau}^0 \partial_{y_1} \left(\frac{d}{ds} y_1 \nu \right) ds -\int_{-\tau}^0 \partial_{y_1} (y_1 \nu)(\delta_0(s) - \delta_{-\tau}(s) ) ds \right. \\
& \left. +  \partial_{y_0} \left\{(\partial_{y_0} U - [\partial_{y_1} U](-\tau)) \nu \right\} 
 - \int_{-\tau}^0 \partial_{y_1} \left\{(\partial_{y_0} U - [\partial_{y_1} U] (-\tau)) \nu \right\} \delta_{-\tau}(s) ds + \frac{1}{2} \sigma^2 \partial_{y_0 y_0} \nu \right) dy\\
= & \partial_t U(t_0, z, \nu_0) - \int_{-\tau}^0 \int_{\R \times \H} D_{\mu_1}U(t_0, z, \nu, y) \frac{d}{ds}y_1 \nu dy ds \\
& + \int_{-\tau}^0 \int_{\R \times \H} D_{\mu_1}U y_1 \nu(\delta_0(s) - \delta_{-\tau}(s)) dy ds
 - \int_{\R \times \H} D_{\mu_0} U(\partial_{y_0} U - [\partial_{y_1} U] (-\tau)) \nu dy\\
& + \int_{-\tau}^0 \int_{\R \times \H} D_{\mu_1} U(\partial_{y_0} U - [\partial_{y_1} U](-\tau) \nu) \delta_{-\tau}(s) dy ds 
+ \int_{\R \times \H} \frac{1}{2} \sigma^2 \partial_{y_0} D_{\mu_0} U \nu dy\\
= & \partial_t U(t_0, z, \nu_0)  + \int_{-\tau}^0  \int_{\R \times \H} y_1 \frac{d}{ds} D_{\mu_1} U \nu dy ds - \int_{\R \times \H} D_{\mu_0}U  (\partial_{y_0} U - [\partial_{y_1} U](-\tau)) \nu dy \\
& +  \int_{\R \times \H} [D_{\mu_1} U] (-\tau) ((\partial_{y_0} U - [\partial_{y_1} U](-\tau)) \nu dy + \frac{1}{2} \sigma^2 \int_{\R \times \H} \partial_{y_0} D_{\mu_0} U \nu dy.
\end{aligned}
\end{equation}
On the other hand, $V$ satisfies the HJB \eqref{hjb} equation. 
\begin{equation}\label{b}
\begin{aligned}
& \partial_t V \\
= & - \frac{1}{2} \sigma^2 \partial_{z_0 z_0} V - \int_{-\tau}^0 z_{1}  \frac{d}{ds} (\partial_{z_{1}} V) ds 
+ \frac{1}{2} \left( \partial_{z_0} V - [\partial_{z_{1}} V] (-\tau)\right)^2 - \frac{\e}{2}\left(\int_{\R} y_0 d \mu_0(y_0) - z_0 \right)^2 \\
= & - \frac{1}{2} \sigma^2 \partial_{z_0 z_0} U - \int_{-\tau}^0 z_{1}  \frac{d}{ds} (\partial_{z_{1}} U) ds 
+ \frac{1}{2} \left( \partial_{z_0} U - [\partial_{z_{1}} U] (-\tau)\right)^2 - \frac{\e}{2}\left(\int_{\R} y_0 d \mu_0(y_0) - z_0 \right)^2. \\
\end{aligned}
\end{equation}
Therefore, subtracting \eqref{b} from \eqref{a}, we have shown that $U$ satisfies the master equation \eqref{master}.
\end{proof}

\subsection{Explicit solution of the master equation}
It turns out that this master equation \eqref{master} can be solved explicitly by making the following ansatz, and  we also define $m_0 := \int_{\R} y_0 d\mu_{0}(y_0)$ and $m_{1} := \int_\H y_{1} d \mu_{1}(y_1)$ for convenience.
\begin{multline}\label{solution}
U(t,z_0, z_1, \nu) = E_0(t) (m_0 - z_0)^2 - 2(m_0 - z_0) \int_{-\tau}^0 E_1(t, -\tau - s) (m_{1} - z_{1}) ds\\
+ \int_{-\tau}^0 \int_{-\tau}^0 E_2(t, -\tau-s, -\tau-r) (m_{1} - z_{1})(m_{1} - z_{1}) ds dr + E_3(t).
\end{multline}
Then, we compute the partial derivatives needed in \eqref{master} explicitly, we have

\begin{equation}
\begin{split}
& \partial_t U 
=  \frac{d E_0(t)}{dt} (m_0 - z_0)^2 
- 2 (m_0 - z_0) \int_{-\tau}^0 \frac{\partial E_1(t, -\tau-s)}{\partial t} (m_{1} - z_{1}) ds \\
& \quad \quad \quad + \int_{-\tau}^0 \int_{-\tau}^0 \frac{\partial E_2(t, -\tau - s, -\tau - r)}{\partial t} (m_{1} - z_{1}) (m_{1}- z_{1}) dsdr + \frac{dE_3(t)}{dt},\\
& \partial_{z_0} U
=  -2 E_0(t) (m_0 - z_0) + 2  \int_{-\tau}^0 E_1(t, -\tau - s) (m_{1} - z_{1}) ds, \\
& \partial_{z_1} U
=  2  E_1(t, -\tau - s) (m_0 - z_0) - 2  \int_{-\tau}^0 E_2(t, -\tau - s, -\tau - r) (m_{1} - z_{1}) dr, \\
& D_{\mu_0} U
=  2 E_0(t) (m_0 - z_0) - 2  \int_{-\tau}^0 E_1(t, -\tau - s) (m_{1} - z_{1}) ds,\\
& D_{\mu_1} U
=  -2  E_1(t, -\tau - s) (m_0 - z_0) + 2  \int_{-\tau}^0 E_2(t, -\tau - s, -\tau - r) (m_{1} - z_{1}) dr,\\
& \partial_{z_0 z_0} U
= 2  E_0(t),
\end{split}
\end{equation}
and plug those into our master equation \eqref{master}. We have 
\begin{align*}
& \frac{d E_0(t)}{dt} (m_0 - z_0)^2 
- 2 (m_0 - z_0) \int_{-\tau}^0 \frac{\partial E_1(t, -\tau - s)}{\partial t} (m_{1} - z_{1}) ds \\
& \quad +  \int_{-\tau}^0 \int_{-\tau}^0 \frac{\partial E_2(t, -\tau - s, -\tau - r)}{\partial t} (m_{1} - z_{1}) ( m_{1}- z_{1}) dsdr + \frac{dE_3(t)}{dt}
+ \frac{1}{2} \sigma^2 (2  E_0(t))\\
& \quad - \int_{-\tau}^0 (m_{1} - z_{1}) \left( 2  \frac{\partial E_1(t, -\tau - s)}{\partial s} (m_0 - z_0) - 2  \int_{-\tau}^0 \frac{\partial E_2(t, -\tau - s, -\tau - r)}{\partial s} (m_{1} - z_{1}) dr\right) ds\\
& \quad - 2 \left((E_0(t) + E_1(t, 0)) (m_0 - z_0)  -  \int_{-\tau}^0 (E_1(t, -\tau - s) + E_2(t, -\tau - s, 0) )(m_{1} - z_{1}) ds \right)^2\\
& \quad + \frac{\e}{2}(m_0 - z_0)^2 = 0.\\
\end{align*}

Collecting $(m_0 - z_0)^2$ terms, $(m_0 - z_0)(m_1 - z_1)$ terms, $(m_{1} - z_{1})^2$ terms, and constant terms, we obtain 
that the function $E_i, i=0,\cdots,3$, satisfy the system of PDEs (\ref{epde2}) with boundary conditions (\ref{boundary2}).
%
%

\section{Convergence of the Nash system}\label{sec:convergence}
From the previous section, we have seen that our master equation is well posed, and we obtained an explicit solution. 
Furthermore, it also describes the limit of Nash equilibria of the $N$-player games as $N\to\infty$. In this section, generalizing  to the case with delay the results of 
\citep{Cardali_Delarue_Lasry_Lions:2015} (see also \citep{kolo}),
we show that the solution of the Nash system \eqref{hjbnplayer} converges to the solution of the master equation \eqref{master} as number of players $N \to +\infty$, with a $1/N$ Cesaro convergence rate.

In Section \ref{sec:mastereq}, we find that \eqref{solution} is a solution to the master equation \eqref{master}.  We set $u^i(t, z_0, z_1) := U(t, z_0^i, z_{1}^i, \nu^i)$, where $\nu^i = \frac{1}{N-1} \sum_{k \neq i} \delta_{(z_0^k, z_{1}^k)}$, denotes the joint empirical measure of $z_0$ and $z_{1}$.  The empirical measure of $z_0$ is given by $\mu_{0}^i = \frac{1}{N-1} \sum_{k \neq i} \delta_{z_0^k}$, and the empirical measure of $z_{1}$ is given by $\mu_{1}^i = \frac{1}{N-1} \sum_{k \neq i} \delta_{z_{1}^k}$. Note that, by direct computation, for $k \neq i$, and any $N \geq 2$,
\begin{equation}\label{uU}
\begin{split}
& \partial_{z_0^k} u^i(t, z_0, z_1) = \frac{1}{N-1} D_{\mu_0^i}  U(t, z_0^i, z_{1}^i, \nu^i, z_0^k),\\
& \partial_{z_1^k} u^i(t, z_0, z_1) = \frac{1}{N-1} D_{\mu_1^i}  U(t, z_0^i, z_{1}^i, \nu^i, z_1^k),\\
& \partial_{z_0^k z_0^k} u^i(t, z_0, z_1) = \frac{1}{N-1} \partial_{z_0^k} [D_{\mu_0^i} U] (t, z_0^i, z_1^i, \nu^i, z_0^k) 
+ \frac{1}{(N-1)^2}D_{\mu_0^i \mu_0^i} U(t, z_0^i, z_1^i, \nu^i, z_0^k, z_0^k).
\end{split}
\end{equation}

\begin{prop}
For any $i \in \{1, \cdots, N\}$, $u^{i}(t, z_0, z_1)$ satisfies
\begin{multline}\label{u}
 \partial_t u^i + \sum_{k=1}^N \frac{1}{2} \sigma^2 \partial_{z_0^k z_0^k} u^i + \sum_{k=1}^N \int_{-\tau}^0 z_{1}^k  \frac{d}{ds} (\partial_{z_{1}^k} u^i) ds 
-\sum_{k\neq i}^N \left(\partial_{z_0^k} u^k - [\partial_{z_{1}^k} u^k](-\tau) \right) \left(\partial_{z_0^k} u^i - [\partial_{z_{1}^k} u^i](-\tau) \right) \\
 - \frac{1}{2} \left( \partial_{z_0^i} u^i - [\partial_{z_{1}^i} u^i] (-\tau)\right)^2 + \frac{\e}{2}(\bar{z}_0 - z_0^i)^2 + e^i(t,z) = 0,
\end{multline}
where $\|e^i(t,z)\| < \frac{C}{N}$, with terminal condition $u^i(T,z) = \frac{c}{2}(\bar{z}_0 - z_0^i)^2$.

This shows that $(u^i)_{i \in \{1, \dots, N\}}$ is ``almost'' a solution to the Nash system \eqref{hjbnplayer}.
\end{prop}

\begin{proof}

We  compute each term in the above equation in terms of $U$ using the relationship \eqref{uU}, and we use the fact that $U$ is a solution to the master equation.

\begin{itemize}
\item 
\begin{align*}
& \sum_{k=1}^N \frac{1}{2} \sigma^2 \partial_{z_0^k z_0^k} u^i(t, z_0, z_{1})\\
= & \frac{1}{2} \sigma^2 \partial_{z_0^i z_0^i} u^i(t, z_0, z_{1}) + \sum_{k \neq i} \frac{1}{2} \sigma^2 \partial_{z_0^k z_0^k} u^i (t, z_0,z_{1})\\
= & \frac{1}{2} \sigma^2 \partial_{z_0^i z_0^i} U (t, z_0^i, z_{1}^i, \nu^i) 
+ \frac{1}{2} \sigma^2 \sum_{k \neq i} \frac{1}{N-1} \partial_{z_0^k}[D_{\mu_0^i} U](t, z_0^i, z_{1}^i, \nu^i, z_0^k)\\
& \quad + \frac{1}{2} \sigma^2 \sum_{k \neq i } \frac{1}{(N-1)^2} D_{\mu_0^i \mu_0^i} U (t, z_0^i, z_{1}^i, \nu^i, z_0^k, z_0^k)\\
= & \frac{1}{2} \sigma^2 \partial_{z_0^i z_0^i} U (t, z_0^i, z_{1}^i, \nu^i) 
+ \frac{1}{2} \sigma^2 \int_\R \partial_{y_0}[D_{\mu_0^i} U](t, z_0^i, z_{1}^i, \nu^i, y_0) d\mu_{0}^i(y_0) \\
 & \quad + \frac{1}{2} \sigma^2 \frac{1}{N-1} \int_\R D_{\mu_0^i \mu_0^i} U(t, z_0^i, z_{1}^i, \nu^i, y_0, y_0) d\mu_{0}^i(y_0). 
\end{align*}

\item 
\begin{align*}
& \sum_{k=1}^N \int_{-\tau}^0 z_{1}^k \frac{d}{ds} (\partial_{z_{1}^k} u^i) ds\\
= & \int_{-\tau}^0 z_{1}^i \frac{d}{ds} (\partial_{z_{1}^i} u^i) ds + \sum_{k \neq i} \int_{-\tau}^0 z_{1}^k \frac{d}{ds} (\partial_{z_{1}^k} u^i) ds\\
= & \int_{-\tau}^0 z_{1}^i \frac{d}{ds} (\partial_{z_{1}^i} U) ds + \sum_{k \neq i} \int_{-\tau}^0 z_{1}^k \frac{d}{ds} \left[ \frac{1}{N-1} D_{\mu_{1}^i} U (t, z_0^i, z_{1}^i, \nu^i, z_1^k) \right] ds\\
= & \int_{-\tau}^0 z_{1}^i \frac{d}{ds} (\partial_{z_{1}^i} U) ds + \int_{-\tau}^0 \int_\H y_{1} \frac{d}{ds} \left[ D_{\mu_{1}^i} U(t, z_0^i, z_{1}^i, \nu^i, y_{1})\right] d \mu_{1}^i(y_1) ds.
\end{align*}

\item  From the solution \eqref{solution} of the master equation, $\partial_{z}U$  is Lipschitz with respect to the measures. Namely,
\begin{equation}
\begin{split}
& |\partial_{z_0} U(t, z^k, \nu^i) - \partial_{z_0} U(t, z^k, \nu^k)| \leq C_1 (d_{MK}(\mu_0^i, \mu_0^k) + d_{MK} (\mu_1^i, \mu_1^k)) \leq \frac{C_1}{N-1},\\
& \|\partial_{z_1} U(t, z^k, \nu^i) - \partial_{z_1} U(t, z^k, \nu^k)\|_\H \leq C_2 (d_{MK}(\mu_0^i, \mu_0^k) + d_{MK} (\mu_1^i, \mu_1^k)) \leq \frac{C_2}{N-1}.\\
\end{split}
\end{equation}
Thus, 
\begin{align*}
& \sum_{k \neq i} (\partial_{z_0^k} u^k - [\partial_{z_{1}^k} u^k] (-\tau)) (\partial_{z_0^k} u^i - [\partial_{z_{1}^k} u^i](-\tau))\\
    = & \sum_{k \neq i} \partial_{z_0^k} U(t, z_0^k, z_{1}^k, \nu^k) \left(\frac{1}{N-1} D_{\mu_0^i} U(t, z_0^i, z_{1}^k, \nu^i, z_0^k) - \frac{1}{N-1} [D_{\mu_1^i} U(t, z_0^i, z_1^i, \nu^i, z_1^k)] (-\tau)\right)\\
& - \sum_{k \neq i} [\partial_{z_1^k}U(t, z_0^k, z_1^k, \nu^k)](-\tau) \left(\frac{1}{N-1} D_{\mu_0^i} U(t, z_0^i, z_{1}^k, \nu^i, z_0^k) - \frac{1}{N-1} [D_{\mu_1^i} U(t, z_0^i, z_1^i, \nu^i, z_1^k)](-\tau)\right)\\
    = & \sum_{k \neq i} \partial_{z_0^k} U(t, z_0^k, z_{1}^k, \nu^i) \left(\frac{1}{N-1} D_{\mu_0^i} U(t, z_0^i, z_{1}^k, \nu^i, z_0^k) - \frac{1}{N-1} [D_{\mu_1^i} U(t, z_0^i, z_1^i, \nu^i, z_1^k)](-\tau)\right)\\
& - \sum_{k \neq i} [\partial_{z_1^k}U(t, z_0^k, z_1^k, \nu^i)](-\tau) \left(\frac{1}{N-1} D_{\mu_0^i} U(t, z_0^i, z_{1}^k, \nu^i, z_0^k) - \frac{1}{N-1} [D_{\mu_1^i} U(t, z_0^i, z_1^i, \nu^i, z_1^k)](-\tau)\right) \\
& + O(\frac{1}{N})\\
= & \int_{-\tau}^0 \int_{\R \times \H} \left( \partial_{y_0}U - \partial_{y_1} U \right) (t, y_0, y_{1}, \nu^i) \cdot\left(D_{\mu_0^i} U-  D_{\mu_{1}^i}U  \right) (t, z_0^i, z_{1}^i, \nu^i, y_0, y_1) d\nu (y_0, y_1) \delta_{-\tau} (s)ds \\
&  + O(\frac{1}{N}).\\
\end{align*}
\end{itemize}

Then,
\begin{align*}
 & \partial_t u^i + \sum_{k=1}^N \frac{1}{2} \sigma^2 \partial_{z_0^k z_0^k} u^i + \sum_{k=1}^N \int_{-\tau}^0 z_{1}^k  \frac{d}{ds} (\partial_{z_{1}^k} u^i) ds \\
 & \quad -\sum_{k\neq i}^N (\partial_{z_0^k} u^k - [\partial_{z_{1}^k}u^k](-\tau)) (\partial_{z_0^k} u^i - [\partial_{z_{1}^k} u^i](-\tau)) - \frac{1}{2} \left( \partial_{z_0^i} u^i - [\partial_{z_{1}^i} u^i] (-\tau) \right)^2 + \frac{\e}{2}(\bar{z}_0 - z_0^i)^2\\
 = &  \partial_t U + \frac{1}{2} \sigma^2 \partial_{z_0^i z_0^i} U (t, z_0^i, z_{1}^i, \nu^i) + \frac{1}{2} \sigma^2 \int_\R \partial_{y_0}[D_{\mu_0^i} U](t, z_0^i, z_{1}^i, \nu^i, y_0) d\mu_{0}^i(y_0) \\
& \quad +  \int_{-\tau}^0 z_{1}^i \frac{d}{ds} (\partial_{z_{1}^i} U) ds + \int_{-\tau}^0 \int_\H y_{1} \frac{d}{ds} \left[ D_{\mu_{1}^i} U(t, z_0^i, z_{1}^i, \nu^i, y_{1})\right] d \mu_{1}^i(y_1) ds\\
& \quad - \int_{\R \times \H} \left( \partial_{y_0} U - [\partial_{y_1} U](-\tau) \right) (t, y_0, y_{1}, \nu^i) \cdot\left(D_{\mu_0^i} U-  [D_{\mu_{1}^i}U] (-\tau)  \right) (t, z_0^i, z_{1}^i, \nu^i, y_0, y_1) d\nu^i (y_0, y_1) \\
& \quad -  \frac{1}{2}(\partial_{z_0^i} U - [\partial_{z_{1}^i} U](-\tau))^2
 + \frac{\e}{2} \left(\int y_0 d\mu_0^i(y_0) - z_0^i \right)^2\\
& \quad + O(\frac{1}{N}) +  \frac{1}{2} \sigma^2 \frac{1}{N-1} \int_\R D_{\mu_0^i \mu_0^i} U(t, z_0^i, z_{1}^i, \nu^i, y_0, y_0) d\mu_{0}^i(y_0) \\
= & O(\frac{1}{N}).
 \end{align*}

\end{proof}

\begin{Theorem}
Let $V^i$ be the solution to the HJB equation \eqref{hjbnplayer} of the $N$-player system, where $N \geq 1$ fixed, and $U$ be the solution to the master equation \eqref{master}. Fix any $(t_0, \nu_0) \in [0, T] \times \cP(\R \times \H)$. Then for any $z \in \R^N$, let $\nu^i = \frac{1}{N-1} \sum_{j \neq i}^N \delta_{(z_0^j, z_1^j)}$, we have
\begin{equation}
\frac{1}{N} \sum_{i=1}^N |V^i(t_0, z) - U(t_0, z^i, \nu^i)| \leq C N^{-1}.
\end{equation}

\end{Theorem}

\begin{proof}
We first apply Ito's formula to $(V^i)_{i \in \{1, \dots, N\}}$, and use the fact that $V^i$ satisfies the HJB equation \eqref{hjbnplayer} for the Nash system. 
\begin{equation} \label{dv}
\begin{aligned}
& d V^i(t, Z_t)\\
= & \partial_t V^i dt + \partial_z V^i d Z_t + \frac{1}{2} Tr( \partial_{zz} V^i d[Z, Z]_t)\\
= & \partial_t V^i dt  + \langle A Z, \partial_{z} V^i \rangle dt + \langle B \hat{\alpha}^i, \partial_{z} V^i \rangle dt + \langle \partial_z V^i, G \rangle dW_t +  \frac{1}{2} Tr( G^* G\partial_{zz} V^i ) dt\\
=& \partial_t V^i dt + \sum_{k=1}^N \int_{-\tau}^0 z_1^k \frac{d}{ds} (\partial_{z_1^k} V^i) ds dt 
- \sum_{k=1}^N (\partial_{z_0^k} V^i - [\partial_{z_1^k} V^i](-\tau))(\partial_{z_0^k} V^k - [\partial_{z_1^k} V^k](-\tau)) dt \\
& \quad + \sum_{k=1}^N \frac{1}{2} \sigma^2 \partial_{z_0^k z_0^k} V^i dt
+ \sum_{k=1}^N \sigma \partial_{z_0^k} V^i  dW_t^k\\
= & \left[-\frac{1}{2} (\partial_{z_0^i} V^i - [\partial_{z_1^i} V^i] (-\tau))^2 - \frac{\e}{2}(\bar{Z}_0 - Z_0^i)^2 \right] dt + \sum_{k=1}^N \sigma \partial_{z_0^k} V^i  dW_t^k.\\
\end{aligned}
\end{equation}

Then, we apply Ito's formula to $u^i(t, Z_t)$, and use the fact that $u$ satisfies \eqref{u}
\begin{equation}\label{du}
\begin{aligned}
& d u^i(t, Z)\\
= & \partial_t u^i dt + \partial_z u^i d Z_t + \frac{1}{2} Tr( \partial_{zz} u^i d[Z, Z]_t)\\
= & \partial_t u^i dt  + \langle A Z, \partial_{z} u^i \rangle dt + \langle B \hat{\alpha}^i, \partial_{z} u^i \rangle dt + \langle \partial_z u^i, G\rangle dt +  \frac{1}{2} Tr( G^* G\partial_{zz} u^i ) dt\\
=& \partial_t u^i dt + \sum_{k=1}^N \int_{-\tau}^0 z_1^k \frac{d}{ds} (\partial_{z_1^k} u^i) ds dt 
- \sum_{k=1}^N (\partial_{z_0^k} u^i - [\partial_{z_1^k} u^i](-\tau))(\partial_{z_0^k} V^k - [\partial_{z_1^k} V^k](-\tau)) dt \\
& \quad + \sum_{k=1}^N \frac{1}{2} \sigma^2 \partial_{z_0^k z_0^k} u^i dt
+ \sum_{k=1}^N \sigma \partial_{z_0^k} u^i  dW_t^k\\
= & \sum_{k=1}^N \left(\partial_{z_0^k} u^k - [\partial_{z_1^k} u^k](-\tau) \right) \left(\partial_{z_0^k} u^i - [\partial_{z_1^k} u^i](-\tau) \right)dt \\
& \quad - \sum_{k=1}^N (\partial_{z_0^k} u^i - [\partial_{z_1^k} u^i](-\tau))(\partial_{z_0^k} V^k - [\partial_{z_1^k} V^k](-\tau)) dt 
- \frac{1}{2} (\partial_{z_0^i} u^i - [\partial_{z_1^i} u^i] (-\tau))^2dt\\
& \quad  - \frac{\e}{2}(\bar{Z}_0 - Z_0^i)^2 dt
- e^i dt
+ \sum_{k=1}^N \sigma \partial_{z_0^k} u^i  dW_t^k.\\
\end{aligned}
\end{equation}

Substracting \eqref{dv} from \eqref{du},  taking the square and applying Ito's formula again, we obtain
\begin{equation}
\begin{aligned}\label{(u-v)^2}
& d[u^i(t, Z_t) - V^i(t, Z_t)]^2\\
= & 2[u^i(t, Z_t) - V^i(t, Z_t)](du^i(t, Z_t) - dV^i(t, Z_t)) + d[u^i - V^i, u^i - V^i]_t\\
= & -2(u^i - V^i) \left(\frac{1}{2} (\partial_{z_0^i} u^i - [\partial_{z_1^i} u^i](-\tau))^2 - \frac{1}{2} (\partial_{z_0^i} V^i - [\partial_{z_1^i} V^i](-\tau))^2 \right) dt
- 2 (u^i - V^i) e^i dt \\
& \quad - 2(u^i - V^i) \left(\sum_{k=1}^N \left(\partial_{z_0^k} u^i - [\partial_{z_1^k} u^i](-\tau) \right)  \left( (\partial_{z_0^k} V^k - \partial_{z_0^k}u^k) - \left([\partial_{z_1^k} V^k](-\tau) - [\partial_{z_1^k} u^k] (-\tau) \right) \right) \right) dt\\
& \quad + \sum_{k=1}^N \sigma^2 |\partial_{z_0^k} u^i - \partial_{z_0^k} V^i|^2 dt 
+ \sum_{k=1}^N \sigma \left( \partial_{z_0^k} u^i - \partial_{z_0^k} V^i \right) dW_t^k.
\end{aligned}
\end{equation}

Recall that $\partial_{z_0^k} u^i (t, z_0, z_1) = \frac{1}{N-1} D_{\mu_0^i} U(t, z_0^i, z_1^i, \nu^i, z_0^k)$ is bounded by $\frac{C}{N}$ for $k \neq i$, and $e^i$ is bounded by $\frac{C}{N}$.  Let $(\Xi^i)_{i \in \{1, \dots, N\}}$ be a family of independent random variable with common law $\nu_0$. By integrating \eqref{(u-v)^2} from $t$ to $T$, and taking expectation conditional on $\Xi$, we have
\begin{equation}
\begin{aligned}
& \E^{\Xi}[|u^i_t - V^i_t|^2] + \sigma^2 \sum_{k=1}^N \E^\Xi \left[\int_t^T |\partial_{z_0^k} u^i_s - \partial_{z_0^k} V^i_s|^2 ds \right] \\
& \quad \quad + C \E^{\Xi} \left[\int_t^T |u_s^i - V_s^i| \cdot | [\partial_{z_1^i} u_s^i](-\tau) - [\partial_{z_1^i} V^i_s](-\tau) | \right]ds\\
 & \quad \quad + \frac{C}{N} \sum_{k=1,k\neq i}^N \E^{\Xi} \left[\int_t^T |u_s^i - V_s^i|\cdot |[\partial_{z_1^k} u_s^k](-\tau) - [\partial_{z_1^k} V_s^k](-\tau) | \right] ds
\\
\leq & \E^{\Xi} [|u^i_T - V^i_T|^2] + C \E^{\Xi} \left[\int_t^T |u_s^i - V_s^i| \cdot|\partial_{z_0^i} u_s^i - \partial_{z_0^i} V^i_s| \right]ds\\
& \quad \quad + \frac{C}{N} \sum_{k=1,k\neq i}^N  \E^\Xi \left[\int_t^T |u_s^i - V_s^i|\cdot |\partial_{z_0^k} u_s^k - \partial_{z_0^k} V_s^k| \right] ds\\
& \quad \quad + \frac{C}{N} \int_t^T \E^\Xi[|u_s^i - V_s^i|]ds.\\
\end{aligned}
\end{equation}
By the fact that $u^i_T = V_T^i$, and using  Young's inequality, we have 
\begin{equation}
\begin{aligned}
& \E^{\Xi} [|u_t^i - V_t^i|^2] + \E^{\Xi} \left[\int_{t}^T |\partial_{z_0^i} u_s^i - \partial_{z_0^i} V_s^i|^2 ds \right]\\
\leq & 0 + \frac{C}{2 \e_1} \E^{\Xi} \left[\int_t^T |u_s^i - V_s^i|^2 ds \right] + \frac{C \e_1}{2} \E^{\Xi}\left[\int_{t}^T |\partial_{z_0^i} u_s^i - \partial_{z_0^i} V_s^i|^2 ds \right] + \frac{C}{2 N \e_2} \sum_{k=1}^N \E^{\Xi} \left[ \int_t^T |u_s^i - V_s^i|^2 ds\right]\\
& + \frac{C \e_2}{ 2N} \sum_{k=1}^N \E^{\Xi} \left[\int_t^T |\partial_{z_0^k} u_s^k - \partial_{z_0^k} V_s^k|^2  ds\right]
+ \frac{C}{2N \e_3} \E^{\Xi} \left[ \int_t^T |u_s^i - V_s^i|^2 ds \right] + \frac{C \e_3}{2N} \int_t^T 1 ds\\ 
\leq & \frac{C}{N^2} + C  \E^{\Xi} \left[\int_t^T |u_s^i - V_s^i|^2 ds \right] + \frac{C}{2N} \sum_{k=1}^N \E^{\Xi} \left[ \int_t^T  |\partial_{z_0^k} u_s^k - \partial_{z_0^k} V_s^k|^2 ds \right].\\
\end{aligned}
\end{equation}
Taking average on both sides, we have
\begin{equation}
\begin{aligned}
& \frac{1}{N} \sum_{i= 1}^N \E^{\Xi}[|u_t^i - V_t^i|^2] + \frac{1}{N} \sum_{i=1}^N \E^{\Xi} \left[ \int_t^T |\partial_{z_0^i} u_s^i - \partial_{z_0^i} V_s^i|^2 ds \right]\\
\leq & \frac{C}{N^2} + \frac{1}{N}\sum_{i=1}^N C \E^{\Xi} \left[ \int_t^T |u_s^i - V_s^i|^2 ds \right] + \frac{C}{2N} \sum_{k=1}^N \E^{\Xi} \left[ \int_t^T  |\partial_{z_0^k} u_s^k - \partial_{z_0^k} V_s^k|^2 ds \right] \\
\Rightarrow & \frac{1}{N} \sum_{i= 1}^N \E^{\Xi}[|u_t^i - V_t^i|^2] \leq \frac{C}{N^2} + C \E^{\Xi} \left[ \frac{1}{N} \sum_{i=1}^N \int_t^T |u_s^i - V_s^i|^2 ds\right].
\end{aligned}
\end{equation}
By Gronwall's inequality and taking supremum over $[0,T]$, we have
\begin{equation}
\sup_{t \in [0, T]} \left[ \frac{1}{N} \sum_{i=1}^N \E^{\Xi}|u_t^i - V_t^i|^2\right] \leq \frac{C}{N^2},
\end{equation}
which implies 
\begin{equation}
\frac{1}{N} \sum_{i=1}^N |u^i(t_0, \Xi) - V^i(t_0, \Xi)| \leq \frac{C}{N}.
\end{equation}
Choosing $\Xi$ uniformly distributed in $(\R \times \H)^N$, then by continuity of $u^i$ and $V^i$, and the fact that $u^i(t, Z)$ is defined by $U(t, Z_0^i, Z_1^i, \nu^i)$, we have, for any $z \in (\R \times \H)^N$,
\begin{equation}
\frac{1}{N} \sum_{i=1}^N |U(t_0, z^i, \nu_0^i) - V^i(t_0, z)| \leq \frac{C}{N}.
\end{equation}
\end{proof}

\section{Conclusion}\label{sec:conclusion}
The mean field game system acts as a characteristic of the master equation. The master equation contains all the information in the mean field game system, and it turns the forward-backward PDE into a single equation. The solution to the mean field game system is a pair $(V, \nu)$, that is the value function and the joint law of current state and past law. The solution to the master equation is a function of $(t, z, \nu)$.

Since our model is linear quadratic, we are able to solve both the mean field game system and the master equation as shown in Section \ref{sec:MFG} and Section \ref{sec:mastereq}, however, the techniques are not the same. The technique for solving the mean field game is that we first make an ansatz for the solution of the HJB equation. Then plugging this ansatz into the Fokker-Planck equation \eqref{fp}, we find that the means of state and past control are constant. Hence, the ansatz \eqref{ansatzmfg} can be verified. On the other hand,  a notion of derivative with respect to measure is needed in order to solve the master equation. Again, we make an ansatz \eqref{solution}, which has a similar form as \eqref{ansatzmfg} but is a function of $(t, z, \nu)$, and we verify that it satisfies the master equation. 

The sets of PDEs \eqref{epde2} with boundary conditions  \eqref{boundary2} are the same for the two problems. This is due to the fact that our model is linear-quadratic and the means of states and past controls are constants. 

Last but not the least, the Nash equilibrium of the corresponding $N$-player game is presented in Section \ref{sec:gamedelay}. The value function \eqref{vi} looks similar to the value function \eqref{ansatzmfg} in the mean field game system and the solution \eqref{solution} to the master equation. As $N \to \infty$, the set of PDEs \eqref{epde1} becomes the same as \eqref{epde2}. This implies that the solution to the mean filed game appears to be the limit of the Nash system, but generally, the convergence has been known in very few specific situations. Additionally,  the solution to the master equation is also a limit to the Nash system, as shown in Section \ref{sec:convergence}. 

To summarize, we have extended the notion of master equation in the context of our toy model with delay, and we have shown that, as in the case without delay, this master equation provides an approximation to the corresponding finite-player game with delay. A general form of such a result, not necessarily for linear-quadratic games,  is part of our ongoing research.

\appendix

\section{Adjoint operator}
\label{appendix}
Let $\varphi$ be a smooth test function defined on $\R \times \H$. In the following computation, we use  the notation
\[
 \langle \varphi, \nu(t)\rangle
= \int_{\R \times \H} \varphi(z) d\nu(t,z).
\]
If the test function $\varphi$ is of the form $\varphi(z)=\int_{-\tau}^0\psi(z_0, z_1(s))ds$ for a smooth function $\psi$ defined on $\R^2$, then
\[
 \langle \varphi, \nu(t)\rangle=
 \int_{-\tau}^0 \int_{\R \times \R} \psi(z_0,z_1(s)) \nu(t,z_0, z_1(s)) dz_0dz_{1}(s) ds,
 \] 
 where $\nu(t,z_0, z_1(s))$ is understood as a two-dimensional density.
By abuse of notation, we also use
\[
 \langle \varphi, \nu(t)\rangle=\int_{\R \times \H} \varphi(z) \nu(t,z) dz =
 \int_{-\tau}^0 \int_{\R \times \R} \psi(z) \nu(t,z) dz ds.
\]
Then, we have
\begin{align*}
& \langle\L_t \varphi, \nu(t)\rangle\\
= & \int_{-\tau}^0 \int_{\R \times \R} z_1 \frac{d \partial_{z_1} \varphi(z)}{ds} \nu(t,z) dz ds 
+ \int_{\R \times \H} - (\partial_{z_0} V - [\partial_{z_1} V](-\tau) )\partial_{z_0} \varphi(z) \nu(t,z) dz \\
& \quad - \int_{-\tau}^0 \int_{\R \times \R}   - (\partial_{z_0} V - [\partial_{z_1} V](-\tau)) \partial_{z_1} \varphi(z) \delta_{-\tau}(s) \nu(t,z) dz ds \\
& \quad + \int_{\R \times \H} \frac{1}{2} \sigma^2 \partial_{z_0 z_0} \varphi(z) \nu(t,z) dz\\
= & -\int_{-\tau}^0 \int_{\R \times \R} \frac{d z_1}{ds}\partial_{z_1} \varphi(z) \nu(t,z) dz ds + \int_{-\tau}^0 \int_{\R \times \R} z_1 \partial_{z_1} \varphi(z) \nu(t,z) (\delta_{0}(s) - \delta_{-\tau}(s)) dz ds\\
& \quad + \int_{\R \times \H} \partial_{z_0} \left\{ (\partial_{z_0} V - [\partial_{z_1}V](-\tau)) \nu(t,z) \right\} \varphi(z) dz\\
& \quad - \int_{-\tau}^0 \int_{\R \times \R} \partial_{z_1} \left\{ (\partial_{z_0} V 
- [\partial_{z_1} V](-\tau)) \nu(t,z) \right\} \delta_{-\tau}(s) \varphi(z) dz ds\\
& \quad + \int_{\R \times \H} \frac{1}{2} \sigma^2 \partial_{z_0 z_0} \nu(t,z) \varphi(z) dz\\
= & \int_{-\tau}^0 \int_{\R \times \R} \partial_{z_1} \left(\frac{d z_1}{ds}\nu(t,z) \right) \varphi(z)  dz ds 
- \int_{-\tau}^0 \int_{\R \times \R} \partial_{z_1} (z_1  \nu(t,z)) \varphi(z) (\delta_{0}(s) - \delta_{-\tau}(s)) dz ds\\
& \quad + \int_{\R \times \H} \partial_{z_0} \left\{ ( \partial_{z_0} V - [\partial_{z_1}V](-\tau)) \nu(t,z)\right\} \varphi(z) dz\\
& \quad  - \int_{-\tau}^0 \int_{\R \times \R} \partial_{z_1} \left\{(\partial_{z_0} V 
- [\partial_{z_1} V](-\tau)) \nu(t,z)\right\} \delta_{-\tau}(s) \varphi(z) dz ds\\
& \quad + \int_{\R \times \H} \frac{1}{2} \sigma^2 \partial_{z_0 z_0} \nu(t,z) \varphi(z) dz\\
= & \langle\varphi, \L_t^* \nu(t)\rangle.
\end{align*}

\bibliographystyle{plain}
\bibliography{master}

\begin{thebibliography}{10}

\bibitem{Bensoussan_book:2007}
Alain Bensoussan, Giuseppe~Da Prato, Michel~C. Delfour, and Sanjoy Mitter.
\newblock {\em Representation and Control of Infinite Dimensional Systems}.
\newblock Birkhauser Basel, 2007.

\bibitem{Cardali_Delarue_Lasry_Lions:2015}
Pierre {Cardaliaguet}, Francois {Delarue}, Jean-Michel {Lasry}, and
  Pierre-Louis {Lions}.
\newblock {The master equation and the convergence problem in mean field
  games}.
\newblock {\em ArXiv: 1509.02505}, 2015.

\bibitem{Carmona_Delarue:2014}
Rene {Carmona} and Francois {Delarue}.
\newblock {The Master Equation for Large Population Equilibriums}.
\newblock {\em ArXiv: 1404.4694}, 2014.

\bibitem{Carmona_Delarue_vol:2018}
Rene Carmona and Francois Delarue.
\newblock {\em Probabilistic Theory of Mean Field Games with Applications I \&
  II}.
\newblock Springer International Publishing, 2018.

\bibitem{Carmona_Fouque_Mousavi_Sun:2018}
Rene Carmona, Jean-Pierre Fouque, Seyyed~Mostafa Mousavi, and Li-Hsien Sun.
\newblock Systemic risk and stochastic games with delay.
\newblock {\em Journal of Optimization and Applications (JOTA)}, 2018.

\bibitem{Carmona_Fouque_Sun:2015}
Rene Carmona, Jean-Pierre Fouque, and Li-Hsien Sun.
\newblock Mean field games and systemic risk.
\newblock {\em Communications in Mathematical Sciences}, 13(4):911--933, 2015.

\bibitem{Chassa_Crisan_Delarue:2015}
Jean-Francois {Chassagneux}, Dan {Crisan}, and Francois {Delarue}.
\newblock {A Probabilistic approach to classical solutions of the master
  equation for large population equilibria}.
\newblock {\em ArXiv: 1411.3009}, 2014.

\bibitem{DaPrato_Zabczyk:2008}
Guiseppe Da{\ }Prato and Jerzy Zabczyk.
\newblock {\em Stochastic Equations in Infinite Dimensions}.
\newblock Cambridge University Press, 2008.

\bibitem{Fabbi_Gozzi_Swiech:2017}
Giorgio Fabbri, Fausto Gozzi, and Andrzej Swiech.
\newblock {\em Stochastic Optimal Control in Infinite Dimension}.
\newblock Springer International Publishing, 2017.

\bibitem{Gozzi_Marinelli:2006}
Fausto Gozzi and Carlo Marinelli.
\newblock Stochastic optimal control of delay equations arising in advertising
  models.
\newblock {\em Stochastic PDEs and Applications VII}, 245:133--148, 2006.

\bibitem{kolo}
Vassili {Kolokoltsov}, M.~{Troeva}, and W.~{Yang}.
\newblock {On the rate of convergence for the mean-field approximation of
  controlled diffusions with large number of players}.
\newblock {\em Dyn. Games Appl.}, 4(2), 2014.

\bibitem{Vinter_Kwong:1981}
Richard~B. Vinter and Raymond~H. Kwong.
\newblock The infinite time quadratic control problem for linear systems with
  state and control delays: An evolution equation approach.
\newblock {\em SIAM Journal on Control and Optimization}, 19(1):139--153, 1981.

\end{thebibliography}

\end{document}